\documentclass[a4paper]{amsart}

\usepackage{amsmath,amssymb,amsthm,mathrsfs}
\usepackage{fullpage}
\usepackage[alphabetic]{amsrefs}
\usepackage[all]{xy}
\newdir{ >}{{}*!/-5pt/@{>}} 
\SelectTips{cm}{} 

\title{Cubical and cosimplicial descent}
\author{Bj{\o}rn I. Dundas  and John Rognes}
\address{Department of Mathematics, University of Bergen, Norway}
\email{dundas@math.uib.no} \urladdr{http://folk.uib.no/nmabd}
\address{Department of Mathematics, University of Oslo, Norway}
\email{rognes@math.uio.no} \urladdr{http://folk.uio.no/rognes}

\date{April 13th 2017}

\newtheorem{theorem}{Theorem}[section]
\newtheorem{proposition}[theorem]{Proposition}
\newtheorem{lemma}[theorem]{Lemma}
\newtheorem{corollary}[theorem]{Corollary}
\theoremstyle{definition}
\newtheorem{definition}[theorem]{Definition}
\theoremstyle{remark}

\DeclareMathOperator*{\hocolim}{hocolim}
\DeclareMathOperator*{\colim}{colim}
\DeclareMathOperator{\Ext}{Ext}
\DeclareMathOperator{\hofib}{hofib}
\DeclareMathOperator*{\holim}{holim}
\DeclareMathOperator{\im}{im}
\DeclareMathOperator{\LKan}{LKan}

\DeclareMathOperator{\sk}{sk}
\DeclareMathOperator{\Tot}{Tot}
\DeclareMathOperator{\trc}{trc}
\newcommand{\bC}{\mathbb{C}}

\newcommand{\bN}{\mathbb{N}}
\newcommand{\bQ}{\mathbb{Q}}
\newcommand{\bS}{\mathbb S}
\newcommand{\bT}{\mathbb{T}}
\newcommand{\bZ}{\mathbb{Z}}
\newcommand{\longto}{\longrightarrow}

\newcommand{\sC}{\mathscr{C}}
\newcommand{\sO}{\mathscr{O}}
\newcommand{\sS}{\mathscr{S}}
\newcommand{\TAQ}{T\hspace{-.6mm}AQ}
\newcommand{\TC}{{T\hspace{-.1mm}C}}
\newcommand{\TF}{{T\hspace{-.4mm}F}}
\newcommand{\THH}{{T\hspace{-.5mm}H\hspace{-.5mm}H}}
\newcommand{\TP}{{T\hspace{-.5mm}P}}
\newcommand{\TR}{{T\hspace{-.5mm}R}}
\renewcommand{\:}{\colon}

\begin{document}

\begin{abstract}
We prove that algebraic $K$-theory, topological Hochschild homology and
topological cyclic homology satisfy cubical and cosimplicial descent
at connective structured ring spectra along $1$-connected maps of such
ring spectra.
\end{abstract}

\maketitle

\section{Introduction}

In this paper we extend the techniques used in \cite{Dun97} to prove that
algebraic $K$-theory, topological Hochschild homology and topological
cyclic homology of connective structured ring spectra all satisfy descent
along $1$-connected maps of such ring spectra.

\begin{theorem}[Cubical descent]
Let $R$ be a connective commutative $\bS$-algebra and let $A$ and $B$
be connective $R$-algebras.  Suppose that the unit map $\eta \: R \to B$
is $1$-connected.  Then the functors $F = K$, $\THH$ and $\TC$
satisfy cubical descent at $A$ along $R \to B$, in the sense that
in each case the natural map
$$
\eta \: F(A) \overset{\simeq}\longto \holim_{T \in P} F(X(T))
$$
is an equivalence of spectra.  Here $P$ denotes the partially ordered
set of nonempty finite subsets $T = \{t_0 < \dots < t_q\}$ of $\bN$,
and $X(T) \cong A \wedge_R B \wedge_R \dots \wedge_R B$,
with $(q+1)$ copies of~$B$.
\end{theorem}

When the cubical diagram $T \mapsto X(T)$ arises from a cosimplicial
spectrum $[q] \mapsto Y^q$, the homotopy limit over $P$ can be replaced
by a homotopy limit over the category $\Delta$ of nonempty finite totally
ordered sets $[q] = \{0 < \dots < q\}$.  This happens, for instance,
when the $R$-algebra $B$ is commutative.

\begin{theorem}[Cosimplicial descent]
Let $R$ be a connective commutative $\bS$-algebra, let $A$ be a connective
$R$-algebra, and let $B$ be a connective commutative $R$-algebra.
Suppose that the unit map $\eta \: R \to B$ is $1$-connected.  Then the
functors $F = K$, $\THH$ and $\TC$ satisfy cosimplicial descent at $A$
along $R \to B$, meaning that in each case the natural map
$$
\eta \: F(A) \overset{\simeq}\longto \holim_{[q] \in \Delta} F(Y^q)
$$
is an equivalence of spectra.  Here $Y^q = A \wedge_R B \wedge_R
\dots \wedge_R B$, with $(q+1)$ copies of~$B$.
\end{theorem}

These results will be proved in Theorems~\ref{K-Cartesian},
\ref{THH-TC-Cartesian} and~\ref{K-THH-TC-descent}.  They apply, in
particular, at any connective $\bS$-algebra $A$ along the unit map $\eta
\: \bS \to MU$ for complex bordism.  In the case of a group $\bS$-algebra
$A = \bS[\Gamma]$, this is relevant for Waldhausen's algebraic $K$-theory
$A(X) \simeq K(\bS[\Gamma])$ of the space $X = B\Gamma$.  In the final section
we discuss a program to analyze $K(\bS[\Gamma])$ and $\TC(\bS[\Gamma])$
in terms of $K(Y^\bullet)$ and $\TC(Y^\bullet)$ for $Y^q = \bS[\Gamma]
\wedge MU \wedge \dots \wedge MU$, with $(q+1)$ copies of~$MU$.

\section{Cubical descent}
\label{sec:cubdesc}

\subsection{Cubical diagrams}

We use terminology similar to that in \cite{Goo91}*{\S1}, including the
notions of $k$-Cartesian and $k$-co-Cartesian cubes.  For each integer
$n\ge1$, let $P^n_\eta$ be the set of subsets $T \subseteq \{1, \dots,
n\}$, partially ordered by inclusion, and let $P^n \subset P^n_\eta$
be the partially ordered subset consisting of the nonempty such~$T$.
A functor $X \: P^n_\eta \to \sC$ from $P^n_\eta$ to any category $\sC$
is called an $n$-dimensional cube, or an $n$-cube, in that category.
The restriction of $X$ to $P^n$ is the subdiagram $X|P^n$ obtained by
omitting the initial vertex $X(\varnothing)$ of the $n$-cube.  Given any
functor $F$ from $\sC$ to spectra, the composite functor $F \circ X$
is an $n$-cube of spectra, which we also denote as $F(X)$.  There is a
natural map
$$
\eta_n \: F(X(\varnothing)) \longto \holim_{T \in P^n} F(X(T))
	= \holim_{P^n} F(X)
$$
from the initial vertex of $F(X)$ to the homotopy limit
\cite{BK72}*{Ch.~XI} of the remaining part of the $n$-cube.
We simply write $F(X)$ in place of $F(X|P^n)$ when
it is clear that the restriction over $P^n \subset P^n_\eta$ is intended.
When forming the homotopy limit of a diagram of spectra we implicitly
assume that each vertex has been functorially replaced by a fibrant
spectrum, and dually for homotopy colimits.  This requires that $F$
takes values in a model category of spectra, such as that of \cite{BF78}
or one of those discussed in \cite{MMSS01},
with homotopy category equivalent to the stable homotopy category.
The $n$-cube $F(X)$ is $k$-Cartesian if and only if $\eta_n$ is
a $k$-connected map.  This is equivalent to the $n$-cube being
$(n+k-1)$-co-Cartesian, since the iterated homotopy cofiber of an
$n$-cube of spectra is equivalent to the $n$-fold suspension of its
iterated homotopy fiber.
Consider also the partially ordered set $P_\eta$ of finite subsets $T$ of
$\bN = \{1, 2, 3, \dots\}$, and let $P \subset P_\eta$ be the partially
ordered subset of nonempty such $T$.  A functor $X$ from $P_\eta$ is an
infinite-dimensional cube, or $\omega$-cube.

\begin{definition}
There is a natural map
$$
\eta \: F(X(\varnothing)) \longto \holim_{T \in P} F(X(T))
        = \holim_P F(X)
$$
and a natural equivalence $\holim_P F(X) \simeq \holim_n \, \holim_{P^n}
F(X)$ that connects $\eta$ to $\holim_n \eta_n$.
We say that $F$ satisfies \emph{cubical descent} over $X$ if $\eta$
is an equivalence of spectra.
\end{definition}

For example, if the connectivity of $\eta_n$ grows to infinity with $n$,
then $\eta$ is an equivalence and $F$ satisfies cubical descent over~$X$.
Cubical descent for $F$ over~$X$ ensures that the homotopy type of the
spectrum $F(X(\varnothing))$ is essentially determined by the homotopy
types of the spectra $F(X(T))$ for nonempty finite subsets $T \subset
\bN$.

\subsection{Amitsur cubes}

Let $R$ be a connective commutative $\bS$-algebra, where $\bS$ denotes
the sphere spectrum.  First, let $A$ and $B$ be connective $R$-modules,
and let $\eta \: R \to B$ be a map of $R$-modules.  We can and will
assume that $A$ and $B$ are flat, i.e., $R$-cofibrant as $R$-modules in
the sense of \cite{Shi04}*{Thm.~2.6(1)}.
Let $n\ge1$, and consider the $n$-cube $X^n  = X^n_R(A,B) \: T \mapsto
X^n(T)$ of spectra given by
\begin{equation} \label{XnT}
X^n(T) = A \wedge_R X_{1,T} \wedge_R \dots \wedge_R X_{n,T} \,,
\end{equation}
where $X_{i,T} = B$ for $i \in T$ and $X_{i,T} = R$ for $i \notin T$.
For each inclusion $T' \subseteq T$ among subsets of $\{1, \dots, n\}$,
the map $X^n(T') \to X^n(T)$ is the smash product over $R$ of $id_A$ with
a copy of $id_B$ for each $i \in T'$, a copy of $\eta \: R \to B$ for
each $i \in T \setminus T'$, and a copy of $id_R$ for each $i \notin T$.
Letting $n$ vary, these definitions assemble to specify an $\omega$-cube
$X^\omega$ in spectra, whose restriction over $P^n_\eta \subset P_\eta$
is the $n$-cube $X^n$.
These constructions are homotopy invariant, because of the assumption
that $A$ and $B$ are flat as $R$-modules.

\begin{lemma} \label{X-is-id-cartesian}
Suppose that $\eta \: R \to B$ is $1$-connected.  Then each
$d$-dimensional subcube of the $n$-cube $X^n = X^n_R(A,B)$ is
$d$-Cartesian and $(2d-1)$-co-Cartesian, for every $0 \le d \le n$.
\end{lemma}

\begin{proof}
Let $B/R$ denote the $1$-connected homotopy cofiber of $\eta \:
R \to B$.  The iterated homotopy cofiber of any $d$-dimensional
subcube of $X^n$ is equivalent to the smash product over $R$ of $A$
with $d$ copies of~$B/R$ and $(n-d)$ copies of~$R$ or~$B$.  Hence it
is at least $(2d-1)$-connected.  Thus the $d$-dimensional subcube is
$(2d-1)$-co-Cartesian, which, as we noted earlier, is equivalent to it
being $d$-Cartesian.
\end{proof}

Next, suppose that $A$ and $B$ are connective $R$-algebras, and that $\eta
\: R \to B$ is the unit map of~$B$.  We can assume that $A$ and $B$ are
$R$-cofibrant as $R$-algebras in the sense of \cite{Shi04}*{Thm.~2.6(3)}.
The underlying $R$-modules of $A$ and $B$ are then flat.  We view $R$
as a base, $A$ as the object at which we wish to evaluate a functor, and
$R \to B$ as a covering that induces a covering $A \to A \wedge_R B$.
In this case the $n$-cube $X^n  = X^n_R(A,B) \: T \mapsto X^n(T)$,
defined by the same expression as in~\eqref{XnT}, takes values in the
category of connective $R$-algebras.  For varying $n$, these assemble to
an $\omega$-cube $X^\omega = X^\omega_R(A,B)$.

\begin{definition}
Let $F$ be any functor from connective $R$-algebras to spectra.  We call
$F(X^\omega)$ the \emph{Amitsur cube} for $F$ at $A$ along $\eta \:
R \to B$, by analogy with the algebraic construction in \cite{Ami59}.
When $F$ satisfies cubical descent over $X^\omega$ we say that $F$
satisfies \emph{cubical descent at $A$ along $R \to B$}.
\end{definition}

Cubical descent for $F$ at $A$ along $R \to B$ ensures that $F(A)$
can be recovered from the diagram of spectra $T \mapsto F(X^\omega(T))$
for $T\in P$, having entries of the form $F(A \wedge_R B \wedge_R \dots
\wedge_R B)$ with one or more copies of~$B$.

\subsection{Cubical descent for $K$, $\THH$ and $\TC$}

Let $A \mapsto K(A)$ denote the algebraic $K$-theory functor from
connective $\bS$-algebras to spectra, see~\cite{BHM93}*{\S5} and
\cite{EKMM97}*{Ch.~VI}.

\begin{theorem} \label{K-Cartesian}
Let $R$ be a connective commutative $\bS$-algebra, let $A$ and $B$ be
connective $R$-algebras, and suppose that the unit map $\eta \: R \to
B$ is $1$-connected.

(a)
The $n$-cube $K(X^n) = K(X^n_R(A,B)) \: T \mapsto K(X^n(T))$ is
$(n+1)$-Cartesian, for each $n\ge1$.

(b)
Algebraic $K$-theory satisfies cubical descent at $A$ along $R \to B$.
\end{theorem}

\begin{proof}
By Lemma~\ref{X-is-id-cartesian} the $n$-cube $X^n \: T \mapsto X^n(T)$
has the property that every $d$-dimensional subcube is $d$-Cartesian.
Hence the assertion that the $n$-cube $K(X^n)$ is $(n+1)$-Cartesian is
the content of \cite{Dun97}*{Prop.~5.1}.
In other words, the natural map $\eta_n \: K(A) \to \holim_{P^n} K(X^n)$
is $(n+1)$-connected.  Thus $\eta \: K(A) \to \holim_P K(X^\omega)$
is an equivalence, and $K$ satisfies cubical descent over~$X^\omega$.
\end{proof}

Let $A \mapsto \THH(A)$ denote the topological Hochschild homology
functor from connective $\bS$-algebras to cyclotomic spectra,
see~\cite{BHM93}*{\S3} and \cite{HM97}*{\S2}.  In particular, the circle
group $\bT$ acts naturally on the underlying spectrum of $\THH(A)$, and
for each subgroup $C = C_{p^m} \subset \bT$ of order $p^m$, where $p$
is a prime, there is a homotopy orbit functor $A \mapsto \THH(A)_{hC}$
and a fixed point functor $A \mapsto \THH(A)^C$.  These are related by
a natural homotopy cofiber sequence of spectra
\begin{equation} \label{norm-restriction}
\THH(A)_{hC_{p^m}} \overset{N}\longto
	\THH(A)^{C_{p^m}} \overset{R}\longto
	\THH(A)^{C_{p^{m-1}}}
\end{equation}
called the norm--restriction sequence.  Let $\TR(A; p) = \holim_{R,m}
\THH(A)^{C_{p^m}}$ be the sequential homotopy limit over the $R$-maps.
Let $F \: \THH(A)^{C_{p^m}} \to \THH(A)^{C_{p^{m-1}}}$ be the map
forgetting part of the invariance, and let $\TF(A; p) = \holim_{F,m}
\THH(A)^{C_{p^m}}$ be the sequential homotopy limit over the $F$-maps.
The $R$-maps induce a self-map of $\TF(A; p)$, also denoted~$R$, and
the topological cyclic homology functor $\TC(A; p)$ can be defined as
the homotopy equalizer of $id$ and $R$:
$$
\xymatrix{
\TC(A; p) \ar[r]^{\pi} & \TF(A; p) \ar@<0.5ex>[r]^{id} \ar@<-0.5ex>[r]_{R}
	& \TF(A; p) \,.
}
$$
The (integral) topological cyclic homology of~$A$, denoted $\TC(A)$,
is defined as the homotopy pullback of two maps
$$
\prod_{\text{$p$ prime}} \TC(A; p)_p \longrightarrow
	\prod_{\text{$p$ prime}} \holim_{F,m}\, \THH(A)^{hC_{p^m}}_p
\longleftarrow \THH(A)^{h\bT} \,,
$$
see~\cite{DGM13}*{Def.~6.4.3.1}.  The left hand map is defined in terms of
$\pi \: \TC(A; p) \to \TF(A; p)$ and the comparison maps $\Gamma_m \:
\THH(A)^{C_{p^m}} \to \THH(A)^{hC_{p^m}}$ from fixed points to homotopy
fixed points.  The right hand map is induced by the forgetful maps
$\THH(A)^{h\bT} \to \THH(A)^{hC_{p^m}}$ associated to the inclusions
$C_{p^m} \subset \bT$, for varying $p$ and~$m$.
The subscript ``$p$'' denotes $p$-completion.
For each prime~$p$, the projection $\TC(A) \to \TC(A; p)$ becomes an
equivalence after $p$-completion.

\begin{theorem} \label{THH-TC-Cartesian}
Let $R$ be a connective commutative $\bS$-algebra, let $A$ and $B$ be
connective $R$-algebras, and suppose that the unit map $\eta \: R \to B$
is $1$-connected.  Consider the $n$-cube $X^n = X^n_R(A,B)$, as above.

(a)
Each $d$-dimensional subcube of $\THH(X^n)$ is $d$-Cartesian and
$(2d-1)$-co-Cartesian.

(b)
Each $d$-dimensional subcube of $\THH(X^n)_{hC}$, $\THH(X^n)^C$ and
$\TR(X^n; p)$ is $d$-Cartesian and $(2d-1)$-co-Cartesian, for every $C =
C_{p^m} \subset \bT$.

(c)
Each $d$-dimensional subcube of $\TF(X^n; p)$ and $\TC(X^n; p)$ is
$(d-1)$-Cartesian and $(2d-2)$-co-Cartesian.

(d)
Topological Hochschild homology, $\THH(-)_{hC}$, $\THH(-)^C$, $\TR(-;p)$,
$\TF(-;p)$, $\TC(-;p)$ and (integral) topological cyclic homology all
satisfy cubical descent at $A$ along $R \to B$.
\end{theorem}

\begin{proof}
(a)
The underlying spectrum of $\THH(A)$ is naturally equivalent to the
realization $B^{cy}(A)$ of the cyclic bar construction $[q] \mapsto A
\wedge \dots \wedge A$, with $(q+1)$ copies of~$A$.
Hence there is a natural equivalence
\begin{align*}
\THH(A \wedge_R B \wedge_R \dots \wedge_R B)
&\simeq B^{cy}(A \wedge_R B \wedge_R \dots \wedge_R B) \\
&\cong B^{cy}(A) \wedge_{B^{cy}(R)} B^{cy}(B) \wedge_{B^{cy}(R)} 
\dots \wedge_{B^{cy}(R)} B^{cy}(B) \,,
\end{align*}
where $B^{cy}(R)$ is a connective commutative $\bS$-algebra, $B^{cy}(A)$
and $B^{cy}(B)$ are flat connective $B^{cy}(R)$-modules, and $B^{cy}(\eta)
\: B^{cy}(R) \to B^{cy}(B)$ is a $1$-connected map of $B^{cy}(R)$-modules.
By Lemma~\ref{X-is-id-cartesian} each $d$-dimensional subcube of
$$
\THH(X^n_R(A,B)) \simeq X^n_{B^{cy}(R)}(B^{cy}(A), B^{cy}(B))
$$
is $(2d-1)$-co-Cartesian.

(b)
Each $d$-dimensional subcube of $\THH(X^n)_{hC}$ is at least as
co-Cartesian as the corresponding $d$-dimensional subcube of $\THH(X^n)$,
because homotopy orbits preserve connectivity.  The analogous claim
for the subcubes of $\THH(X^n)^C$, with $C = C_{p^m}$, follows by
induction on~$m$ from the norm--restriction homotopy cofiber sequence.
The equivalent Cartesian claim follows.  The Cartesian claim for the
subcubes of the sequential homotopy limit $\TR(X^n; p)$ follows from
the Milnor $\lim$-$\lim^1$ sequence.  In this case no connectivity
is lost, because $\lim^1$ vanishes on sequences of surjections,
see~\cite{Dun97}*{Lem.~4.3}.

(c)
The Cartesian claims for $\TF(X^n; p)$ and $\TC(X^n; p)$ follow
from those for the cubes $\THH(X^n)^C$, since sequential homotopy
limits and homotopy equalizers reduce connectivity by at most one,
see~\cite{Dun97}*{Prop.~4.4}.

(d)
For each of the functors $F = \THH$, $\THH(-)_{hC}$, $\THH(-)^C$,
$\TR(-;p)$, $\TF(-;p)$ and $\TC(-;p)$, the natural map $\eta_n \: F(A)
\to \holim_{P^n} F(X^n)$ is $n$- or $(n-1)$-connected.  Thus the natural
map $\eta \: F(A) \to \holim_P F(X^\omega)$ is an equivalence, and each
of these functors satisfies cubical descent over~$X^\omega$.  Finally,
for integral topological cyclic homology we use that each vertex in the
natural diagram defining it satisfies cubical descent, since $\holim_P$
commutes up to a natural chain of equivalences with other homotopy limits.
\end{proof}

\subsection{Two spectral sequences}
\label{subsec:cubdescspseq}

For each $\omega$-cube $F(X)$ of spectra, the equivalent homotopy
limits
$$
\holim_P F(X) \simeq \holim_n \, \holim_{P^n} F(X)
$$
give rise to two spectral sequences.
On the one hand, we have the homotopy spectral sequence
$$
E_1^{s,t} = \pi_{t-s} \hofib(p_s)
	\Longrightarrow_s \pi_{t-s} (\holim_n \, \holim_{P^n} F(X))
$$
associated to the tower of fibrations
\begin{equation} \label{P-tower}
\dots \to \holim_{P^{s+1}} F(X) \overset{p_s}\longto
	\holim_{P^s} F(X) \to \dots \,,
\end{equation}
see~\cite{BK72}*{IX.4.2}.  Here $\hofib(p_s)$ denotes the homotopy fiber
of the map $p_s$, which is equivalent to the iterated homotopy fiber of
the $s$-dimensional subcube of $F(X|P^{s+1})$ with vertices
indexed by the $T \in
P^{s+1}$ with $s+1 \in T$.  This spectral sequence is conditionally
convergent \cite{Boa99}*{Def.~5.10} to the sequential limit $\lim_s
\pi_*(\holim_{P^s} F(X))$ of the homotopy groups of that tower.
For each $r\ge1$, the bigrading of the $d_r$-differential is given by
$$
d_r^{s,t} \: E_r^{s,t} \longto E_r^{s+r,t+r-1} \,.
$$
Following the usual conventions for Adams spectral sequences, $d_r$
maps bidegree $(x,y)$ in the $(t-s,s)$-plane to bidegree $(x-1,y+r)$.
If the derived limit $RE_\infty = \lim_r^1 E_r$ of $E_r$-terms vanishes
in each bidegree, then the spectral sequence converges strongly to $\pi_*
(\holim_n \, \holim_{P^n} F(X))$, see~\cite{Boa99}*{Thm.~7.4}, which is
isomorphic to $\pi_* F(X(\varnothing))$ when $F$ satisfies cubical descent
over $X$.  We call this the \emph{cubical descent spectral sequence}.

\begin{corollary} \label{vanishing-line}
Let $R$, $A$, $B$ and $\eta$ be as in Theorem~\ref{K-Cartesian} and
Theorem~\ref{THH-TC-Cartesian}.  When $F$ is one of the functors $K$,
$\THH$, $\dots$, $\TC(-;p)$ or $\TC$ of those theorems, and $X =
X^\omega_R(A,B)$, the cubical descent spectral sequence
$$
E_1^{s,t} = \pi_{t-s} \hofib(p_s)
        \Longrightarrow_s \pi_{t-s} F(A)
$$
vanishes above a line of slope~$+1$ in the $(t-s,s)$-plane, starting
at the $E_1$-term, hence collapses at a finite stage in each bidegree.
Thus $RE_\infty = 0$ and the spectral sequence is strongly convergent.
\end{corollary}

\begin{proof}
Let $c = 1$ when $F$ is $K$, let $c=0$ when $F$ is one of the functors
$\THH$, $\THH(-)_{hC}$, $\THH(-)^C$ or $\TR(-;p)$, and let $c=-1$ when
$F$ is $\TF(-;p)$ or $\TC(-;p)$.  Then $\eta_n \: F(A) \to \holim_{P^n}
F(X^n)$ is $(n+c)$-connected for each $n\ge1$, so $p_s \: \holim_{P^{s+1}}
F(X^{s+1}) \to \holim_{P^s} F(X^s)$ is $(s+c)$-connected for each $s\ge1$.
Thus $E_1^{s,t} = 0$ for $t-s < s+c$ and $s\ge1$.

The case $F = \TC$ remains.  For this we appeal to
\cite{DGM13}*{Thm.~7.0.0.2} (where $K(B)$ in the lower left-hand corner
of the displayed square should be replaced with $K(A)$), to see that
for each $s\ge1$ the square
$$
\xymatrix{
\holim_{P^{s+1}} K(X^{s+1}) \ar[r]^-{p_s} \ar[d]_{\trc}
  &\holim_{P^s} K(X^s) \ar[d]^{\trc} \\
\holim_{P^{s+1}} \TC(X^{s+1}) \ar[r]^-{p_s}
  &\holim_{P^s} \TC(X^s)
}
$$
is homotopy Cartesian.  This uses that $\pi_0 X^{s+1}(T)$ is constant
as a functor of $T$, by our assumptions on $R$, $A$, $B$ and~$\eta$.
Hence the spectral sequences for $K$ and $\TC$ have the same groups
$E_1^{s,t}$ for all $s\ge1$, and from the case $F = K$ we know that
these groups vanish for $t-s < s+1$.
\end{proof}

On the other hand, we have the homotopy limit spectral sequence
$$
E_2^{s,t} = {\lim_P}^s \  \pi_t F(X)
	\Longrightarrow_s \pi_{t-s} \holim_P F(X)
$$
of \cite{BK72}*{XI.7.1}, associated to $P$-shaped diagrams of spectra.
We call this the \emph{cubical homotopy limit spectral sequence}.
The abutment is isomorphic to $\pi_* F(X(\varnothing))$ when $F$
satisfies cubical descent over $X$.  We do not know whether the cubical
descent spectral sequence and the cubical homotopy limit spectral
sequence are isomorphic for all $\omega$-cubical diagrams $F(X)$, but
in Subsection~\ref{subsec:cosimpdescspseq} we will see that this is
the case whenever the diagram $F(X|P)$ arises from a cosimplicial object
$F(Y^\bullet)$ by composition with a specific functor $f \: P \to \Delta$,
which we will now introduce.

\section{Cosimplicial descent}

\subsection{Cosimplicial objects}

Let $\Delta_\eta$ be the category of finite totally ordered sets $[q]
= \{0 < \dots < q\}$ and order-preserving functions, where $q\ge-1$ is
an integer.  The object $[-1] = \varnothing$ is initial in this category.
Let $\Delta \subset \Delta_\eta$ be the full subcategory generated by the
objects $[q]$ with $q\ge0$.  For $n\ge1$, let $\Delta^{<n}_\eta \subset
\Delta_\eta$ and $\Delta^{<n} \subset \Delta$ be the respective full
subcategories generated by the objects $[q]$ with $q<n$.

A coaugmented cosimplicial object in~$\sC$ is a functor $Y \: \Delta_\eta
\to \sC$.  We let $Y^q = Y([q])$, for each $q\ge-1$.  We can also
write $Y$ as $\eta \: Y^{-1} \to Y^\bullet$, where $Y^\bullet$ is the
cosimplicial object given by the restriction of $Y$ over $\Delta \subset
\Delta_\eta$.
For $n\ge1$, a functor $Y_{<n} \: \Delta^{<n}_\eta \to \sC$ is called
a coaugmented $(n-1)$-truncated cosimplicial object, also written
as $\eta \: Y^{-1} \to Y_{<n}^\bullet$, where $Y_{<n}^\bullet =
Y_{<n}|\Delta^{<n}$.

\begin{definition}
Given any functor $F$ from $\sC$ to spectra, the composite functor $F
\circ Y$ is a coaugmented cosimplicial spectrum~$F(Y)$, which we can write
as $\eta \: F(Y^{-1}) \to F(Y^\bullet)$.  We say that $F$ satisfies
\emph{cosimplicial descent} over~$Y$ if the natural map
$$
\eta \: F(Y^{-1}) \longto \holim_{[q] \in \Delta} F(Y^q)
	= \holim_\Delta F(Y^\bullet)
$$
is an equivalence of spectra.  This ensures that the homotopy type
of $F(Y^{-1})$ is essentially determined by the homotopy types of the
spectra $F(Y^q)$ for $q\ge0$.
\end{definition}

\subsection{Comparison of cubical and cosimplicial objects}
\label{subsec:cubvscos}
There is a well-defined functor $f_\eta \: P_\eta \longto \Delta_\eta$
that maps the element $T = \{t_0 < \dots < t_q\} \subset \bN$ to the
object $[q]$, and maps the inclusion $T \setminus \{t_i\} \subset T$
to the $i$-th face operator $\delta^i \: [q-1] \to [q]$, for each $0 \le
i \le q$.
By restricting $f_\eta$, one gets functors $f \: P \to \Delta$, $f^n_\eta
\: P^n_\eta \to \Delta^{<n}_\eta$ and $f^n \: P^n \to \Delta^{<n}$.

Composition with $f_\eta \: P_\eta \to \Delta_\eta$ takes each
coaugmented cosimplicial spectrum $F(Y)$ to an $\omega$-cube $F(X) =
F(Y) \circ f_\eta$ in the category of spectra.  Likewise, composition
with $f^n_\eta \: P^n_\eta \to \Delta^{<n}_\eta$ takes each coaugmented
$(n-1)$-truncated cosimplicial spectrum $F(Y_{<n})$ to an $n$-cube
$F(X^n) = F(Y_{<n}) \circ f^n_\eta$ of spectra.  If $F(Y_{<n})$ is given
by restricting $F(Y)$ over $\Delta^{<n}_\eta \subset \Delta_\eta$ then
$F(X^n)$ is given by restricting $F(X)$ over $P^n_\eta \subset P_\eta$.

\begin{proposition} \label{Delta-vs-P}
The functors $f \: P \to \Delta$ and $f^n \: P^n \to \Delta^{<n}$
are left cofinal.  Hence, for any cosimplicial spectrum $F(Y^\bullet)$
the canonical map
$$
f^* \: \holim_{\Delta} F(Y^\bullet) \overset{\simeq}\longto
	\holim_P \, (F(Y^\bullet) \circ f)
$$
is an equivalence, and for any $(n-1)$-truncated cosimplicial spectrum
$F(Y_{<n}^\bullet)$ the canonical map
$$
f_n^* \: \holim_{\Delta^{<n}} F(Y_{<n}^\bullet) \overset{\simeq}\longto
	\holim_{P^n} \, (F(Y_{<n}^\bullet) \circ f^n)
$$
is an equivalence.
\end{proposition}

\begin{proof}
The assertion that $f^n$ is left cofinal, i.e., that the left fiber
$f_n/[q]$ has contractible nerve for each object $[q]$ of $\Delta^{<n}$,
is proved in \cite{Car08}*{\S6}.  Note that the argument offered at
this point in \cite{DGM13}*{A8.1.1} is flawed.  The left fiber $f/[q]$
is the increasing union of the left fibers $f_n/[q]$, hence its nerve
is also contractible.  The equivalences of homotopy limits then follow
from the cofinality theorem of \cite{BK72}*{XI.9.2}.
\end{proof}

\begin{corollary} \label{cosimp-vs-cube}
Let $Y$ be a coaugmented cosimplicial object in $\sC$, and let $F$
be a functor from $\sC$ to spectra.  Then $F$ satisfies cosimplicial
descent over $Y$ if and only if $F$ satisfies cubical descent over $X
= Y \circ f_\eta$.
\end{corollary}


\begin{corollary} \label{f-fn-lims-iso}
For any cosimplicial abelian group $A^\bullet$, the canonical homomorphisms
$$
f^* \: \pi^s A^\bullet \cong {\lim_\Delta}^s A^\bullet
	\overset{\cong}\longto {\lim_P}^s \, (A^\bullet \circ f)
\qquad\text{and}\qquad
f_n^* \: {\lim_{\Delta^{<n}}}^s A^\bullet
	\overset{\cong}\longto {\lim_{P^n}}^s \, (A^\bullet \circ f^n)
$$
are isomorphisms, for each $s\ge0$.
\end{corollary}

\begin{proof}
See \cite{BK72}*{XI.7.2 and XI.7.3}.
\end{proof}


\subsection{Amitsur resolutions}
\label{subsec:amres}

We now suppose that $R$ is a connective commutative $\bS$-algebra, that
$A$ is a connective $R$-algebra, and that $B$ is a connective commutative
$R$-algebra with unit map $\eta \: R \to B$ and multiplication map $\mu \:
B \wedge_R B \to B$.  The condition that $B$ is commutative ensures that
$\mu$ is a morphism of connective $R$-algebras.  We can assume that $A$
is $R$-cofibrant as an $R$-algebra, and that $B$ is $R$-com-alg-cofibrant
as a commutative $R$-algebra in the sense of \cite{Shi04}*{Thm.~3.2}.
(This follows if $B$ is positive stable cofibrant in the sense of
\cite{MMSS01}*{Thm.~15.2(i)}.)
The underlying $R$-modules of $A$ and~$B$ are then flat, by
\cite{Shi04}*{Cor.~4.3}, so that the results of Section~\ref{sec:cubdesc}
carry over to this situation.

Consider the cosimplicial connective $R$-algebra $Y^\bullet =
Y^\bullet_R(A,B) \: [q] \mapsto Y^q$ given by
\begin{equation} \label{YRAB}
Y^q = A \wedge_R B \wedge_R \dots \wedge_R B \,,
\end{equation}
with $(q+1)$ copies of~$B$.  The coface maps
$$
d^i = id_A \wedge (id_B)^{\wedge i} \wedge \eta
	\wedge (id_B)^{\wedge q-i} \: Y^{q-1} \longto Y^q
$$
for $0 \le i \le q$, and the codegeneracy maps
$$
s^j = id_A \wedge (id_B)^{\wedge j} \wedge \mu
	\wedge (id_B)^{\wedge q-j} \: Y^{q+1} \longto Y^q
$$
for $0 \le j \le q$, are induced by the unit map~$\eta$ and the
multiplication map~$\mu$, respectively.  The unit map~$\eta$ also
induces a coaugmentation $\eta \: A \to Y^0 = A \wedge_R B$ that makes
$A \to Y^\bullet$ a coaugmented cosimplicial connective $R$-algebra,
i.e., a functor $Y = Y_R(A,B)$ from $\Delta_\eta$ to the category of
connective $R$-algebras:
$$
\xymatrix{
A \ar[r]^-{\eta}
& A \wedge_R B \ar@<0.8ex>[r] \ar@<-0.8ex>[r]
	& A \wedge_R B \wedge_R B \ar[l] \ar@<1.6ex>[r] \ar[r]
	\ar@<-1.6ex>[r] & \ \cdots \ar@<0.8ex>[l] \ar@<-0.8ex>[l] \,.
}
$$

\begin{definition}
Given a functor $F$ from connective $R$-algebras to spectra, we call
the coaugmented cosimplicial spectrum $\eta \: F(A) \to F(Y^\bullet) =
F(Y^\bullet_R(A,B))$ the \emph{Amitsur resolution} for $F$ at $A$ along
$\eta \: R \to B$.  When $F$ satisfies cosimplicial descent over $A \to
Y^\bullet$ we say that $F$ satisfies \emph{cosimplicial descent at $A$
along $R \to B$}.
\end{definition}

Cosimplicial descent for $F$ at $A$ along $R \to B$ ensures that
the coaugmented cosimplicial spectrum
$$
\xymatrix{
F(A) \ar[r]^-{\eta}
& F(A \wedge_R B) \ar@<0.8ex>[r] \ar@<-0.8ex>[r]
        & F(A \wedge_R B \wedge_R B) \ar[l] \ar@<1.6ex>[r] \ar[r]
        \ar@<-1.6ex>[r] & \ \cdots \ar@<0.8ex>[l] \ar@<-0.8ex>[l]
}
$$
induces an equivalence
$$
\eta \: F(A) \overset{\simeq}\longto \holim_{[q] \in \Delta} F(Y^q)
	= \holim_\Delta F(Y^\bullet)
$$
from $F(A)$ to the homotopy limit of the remainder of the diagram,
having entries of the form $F(A \wedge_R B \wedge_R \dots \wedge_R B)$
with one or more copies of~$B$.

\begin{lemma} \label{X-is-Yf}
There is a natural isomorphism $X^\omega_R(A, B) \cong Y_R(A, B)
\circ f_\eta$.
\end{lemma}

\begin{proof}
For $T = \{t_0 < \dots < t_q\} \subset \bN$, with $f_\eta(T) = [q]$,
both
$X^\omega_R(A, B)(T)$ and $Y_R(A, B)([q])$ are identified with
$$
A \wedge_R B \wedge_R \dots \wedge_R B \,,
$$
where there are $(q+1)$ copies of~$B$.  For each $T' \subset T$, the
induced morphisms in $X^\omega_R(A, B)$ and $Y_R(A, B) \circ f_\eta$
are evidently compatible with these identifications.
\end{proof}

\begin{theorem} \label{K-THH-TC-descent}
Let $R$ be a connective commutative $\bS$-algebra, $A$ a connective
$R$-algebra and $B$ a connective commutative $R$-algebra, and suppose
that the unit map $\eta \: R \to B$ is $1$-connected.  Then algebraic
$K$-theory, topological Hochschild homology, $\THH(-)_{hC}$, $\THH(-)^C$,
$\TR(-;p)$, $\TF(-;p)$, $\TC(-;p)$, 
and (integral) topological cyclic homology all satisfy cosimplicial
descent at $A$ along $R \to B$.
\end{theorem}

\begin{proof}
Combine Theorems~\ref{K-Cartesian} and~\ref{THH-TC-Cartesian},
Corollary~\ref{cosimp-vs-cube} and Lemma~\ref{X-is-Yf}.
\end{proof}

\subsection{More spectral sequences}
\label{subsec:cosimpdescspseq}

For each coaugmented cosimplicial spectrum $\eta \: F(Y^{-1}) \to
F(Y^\bullet)$, the equivalent homotopy limits
$$
\holim_\Delta F(Y^\bullet) \simeq 
	\holim_n \, \holim_{\Delta^{<n}} F(Y^\bullet)
$$
give rise to two spectral sequences.  These turn out to be
isomorphic to one another, as well as to the two spectral sequences of
Subsection~\ref{subsec:cubdescspseq}, when we consider cubical diagrams
that arise from cosimplicial diagrams by composition with the left
cofinal functor~$f$.

On the one hand, we have the homotopy spectral sequence
$$
E_1^{s,t} = \pi_{t-s} \hofib(\delta_s)
        \Longrightarrow_s
	\pi_{t-s}(\holim_n \, \holim_{\Delta^{<n}} F(Y^\bullet))
$$
associated to the tower of fibrations
\begin{equation} \label{Delta-tower}
\dots \to \holim_{\Delta^{<s+1}} F(Y^\bullet) \overset{\delta_s}\longto
        \holim_{\Delta^{<s}} F(Y^\bullet) \to \dots \,,
\end{equation}
which we call the \emph{cosimplicial descent spectral sequence}.

By Proposition~\ref{Delta-vs-P}, the tower of
fibrations~\eqref{Delta-tower} is equivalent to the tower~\eqref{P-tower}
when $X = Y \circ f_\eta$, $F(X) = F(Y) \circ f_\eta$ and $F(X|P)
= F(Y^\bullet) \circ f$.  Hence in these cases the cosimplicial descent
spectral sequence for $\eta \: F(Y^{-1}) \to F(Y^\bullet)$ is isomorphic
to the cubical descent spectral sequence for $F(X)$.

There is also the \emph{Bousfield--Kan homotopy spectral sequence}
$$
E_1^{s,t} = \pi_{t-s} \hofib(\tau_s)
	\Longrightarrow_s \pi_{t-s} \Tot F(Y^\bullet)
$$
of the cosimplicial spectrum $F(Y^\bullet)$, associated to
the tower of maps
\begin{equation} \label{Tot-tower}
\dots \to \Tot_s F(Y^\bullet) \overset{\tau_s}\longto
	\Tot_{s-1} F(Y^\bullet) \to \dots \,,
\end{equation}
see~\cite{BK72}*{X.6.1}.  To ensure that the maps $\tau_s$ are fibrations,
we first implicitly replace $F(Y^\bullet)$ with an equivalent fibrant
cosimplicial spectrum $F(Y^\bullet)'$ \cite{BK72}*{X.4.6}.  This fibrant
replacement does not change the homotopy type of $F(X|P) = F(Y^\bullet)
\circ f$ or any of the associated homotopy limits that we consider.
The $E_2$-term of this homotopy spectral sequence can then be expressed as
$$
E_2^{s,t} = \pi^s \pi_t F(Y^\bullet)
	\Longrightarrow_s \pi_{t-s} \Tot F(Y^\bullet) \,,
$$
see~\cite{BK72}*{X.7.2}.
We prove in Proposition~\ref{tot-n-as-holim} that there are natural
equivalences
$$
\Tot_n F(Y^\bullet)
	\overset{\simeq}\longto \holim_{\Delta^{<n+1}} F(Y^\bullet)
$$
(after the implicit fibrant replacement of $F(Y^\bullet)$), compatible
for varying~$n$ with the equivalence
$$
\Tot F(Y^\bullet)
        \overset{\simeq}\longto \holim_{\Delta} F(Y^\bullet)
$$
of \cite{BK72}*{XI.4.4}.  Hence the tower~\eqref{Tot-tower} is equivalent
to the tower~\eqref{Delta-tower}, and the cosimplicial descent spectral
sequence is isomorphic to the Bousfield--Kan homotopy spectral sequence
of $F(Y^\bullet)$, starting with the $E_1$-term.

On the other hand, we have the \emph{cosimplicial homotopy limit spectral
sequence}
$$
E_2^{s,t} = {\lim_\Delta}^s \, \pi_t F(Y^\bullet)
	\Longrightarrow_s \pi_{t-s} \holim_\Delta F(Y^\bullet) \,.
$$
By \cite{BK72}*{X.7.5} there is a natural map from the Bousfield--Kan
homotopy spectral sequence of $F(Y^\bullet)$ to the cosimplicial homotopy
limit spectral sequence, and this map is an isomorphism at the $E_2$-term,
hence also at all later terms.

Finally, the functor $f \: P \to
\Delta$ induces a map $f^*$ of homotopy limit spectral sequences from
$$
E_2^{s,t} = {\lim_\Delta}^s \, \pi_t F(Y^\bullet)
	\Longrightarrow_s \pi_{t-s} \holim_\Delta F(Y^\bullet)
$$
to
$$
E_2^{s,t} = {\lim_P}^s \, \pi_t F(X)
	\Longrightarrow_s \pi_{t-s} \holim_P F(X) \,,
$$
where $F(X|P) = F(Y^\bullet) \circ f$.
This is the map of Bousfield--Kan spectral sequences associated to a
canonical map $f^* \: \Pi^\bullet_\Delta F(Y^\bullet) \to \Pi^\bullet_P
F(X|P)$ of cosimplicial spectra, see~\cite{BK72}*{XI.5.1}, so the map of
$E_2$-terms is the isomorphism $f^*$ of Corollary~\ref{f-fn-lims-iso}.
Hence in these cases the cosimplicial homotopy limit spectral sequence
for $F(Y^\bullet)$ is isomorphic to the cubical homotopy limit spectral
sequence for $F(X|P)$.

\begin{proposition}
(a)
Let $F(Y^\bullet)$ be any cosimplicial spectrum, with fibrant replacement
$F(Y^\bullet)'$, and let $F(X|P) = F(Y^\bullet) \circ f$.  The first
three of the spectral sequences
\begin{alignat*}{2}
E_1^{s,t} &= \pi_{t-s} \hofib(p_s) \Longrightarrow_s
	\pi_{t-s}(\holim_n \, \holim_{P^n} F(X))
\qquad && \text{(cubical descent)} \\
E_1^{s,t} &= \pi_{t-s} \hofib(\delta_s) \Longrightarrow_s
	\pi_{t-s}(\holim_n \, \holim_{\Delta^{<n}} F(Y^\bullet))
\qquad && \text{(cosimplicial descent)} \\
E_1^{s,t} &= \pi_{t-s} \hofib(\tau_s) \,,\,
E_2^{s,t} = \pi^s \, \pi_t F(Y^\bullet) \Longrightarrow_s
        \pi_{t-s} \Tot F(Y^\bullet)'
\qquad && \text{(Bousfield--Kan)} \\
E_2^{s,t} &= {\lim_\Delta}^s \, \pi_t F(Y^\bullet) \Longrightarrow_s
	\pi_{t-s} \holim_\Delta F(Y^\bullet)
\qquad && \text{(cosimplicial homotopy limit)} \\
E_2^{s,t} &= {\lim_P}^s \, \pi_t F(X) \Longrightarrow_s
        \pi_{t-s} \holim_P F(X)
\qquad && \text{(cubical homotopy limit)}
\end{alignat*}
are isomorphic at the $E_1$-term, and all five are isomorphic at the
$E_2$-term, hence also at all later terms.  If $RE_\infty = 0$ these
spectral sequences all converge strongly to the indicated abutments.

(b) 
Let $F(Y)$ be any coaugmented cosimplicial spectrum, and let $F(X)
= F(Y) \circ f_\eta$.  If $F$ satisfies cubical descent over $X$, or
equivalently, if $F$ satisfies cosimplicial descent over $Y$, then each
abutment in~(a) is isomorphic to $\pi_* F(Y^{-1}) = \pi_*
F(X(\varnothing))$.  If $F(Y)$ is the Amitsur resolution $\eta \:
F(A) \to F(Y^\bullet_R(A,B))$, so that $F(X)$ is the Amitsur cube
$F(X^\omega_R(A,B))$, then this common abutment is $\pi_* F(A)$.
\end{proposition}

\begin{proof}
(a)
The stated isomorphisms were all discussed before the statement of
the proposition.  Each spectral sequence is derived from a tower of
fibrations, hence is conditionally convergent to the sequential limit
of the homotopy groups of the terms in this tower.  When $RE_\infty$
vanishes, each homotopy group of the homotopy limit of the tower
is isomorphic to that sequential limit, and the spectral sequence is
strongly convergent, by \cite{BK72}*{IX.5.4} or \cite{Boa99}*{Thm.~7.4}.

(b)
Cosimplicial descent for $F$ over~$Y$ ensures that $F(Y^{-1}) \simeq
\holim_\Delta F(Y^\bullet)$.
\end{proof}

\begin{theorem} \label{descentspseq}
Let $R$ be a connective commutative $\bS$-algebra, $A$ a connective
$R$-algebra, $B$ a connective commutative $R$-algebra and suppose that
the unit map $\eta \: R \to B$ is $1$-connected.
Let $F$ be one of the functors
\begin{itemize}
\item $K$, $\TC$ (with $c=+1$),
\item $\THH$, $\THH(-)_{hC}$, $\THH(-)^C$, $\TR(-;p)$ (with $c=0$),
\item $\TF(-;p)$, $\TC(-;p)$ (with $c=-1$).
\end{itemize}
Then the $E_1$-term of the (cubical/cosimplicial) descent spectral sequence
for $F$ at $A$ along $\eta \: R \to B$ vanishes in all bidegrees $(s,t)$
with $t-s < s+c$ and $s\ge1$.  It is strongly convergent to $\pi_* F(A)$,
with $E_2$-term given by
$$
E_2^{s,t} = \pi^s \, \pi_t F(Y^\bullet_R(A,B))
	\Longrightarrow_s \pi_{t-s} F(A) \,.
$$
\end{theorem}

\begin{proof}
The description of the $E_2$-term is that of the Bousfield--Kan spectral
sequence.  The vanishing line for the $E_1$-term of the cubical descent
spectral sequence is that of Corollary~\ref{vanishing-line}.  It follows
that in each bidegree~$(s,t)$ there is a finite $r$ such that $E_r^{s,t}
= E_\infty^{s,t}$, which implies that $RE_\infty^{s,t} = 0$.  Hence 
each version of the spectral sequence converges strongly to
the homotopy groups of $F(Y^{-1}_R(A,B)) = F(A)$.
\end{proof}

In the discussion above we used the following result, for which a proof
does not seem to have appeared in the literature.

\begin{proposition} \label{tot-n-as-holim}
Let $Z^\bullet$ be any fibrant cosimplicial space.
There are compatible natural equivalences
$$
z^* \: \Tot_n Z^\bullet \overset{\simeq}\longto
	\holim_{\Delta^{<n+1}} Z^\bullet
$$
for all $0 \le n \le \infty$.
\end{proposition}

\begin{proof}
When $n=\infty$, this result is \cite{BK72}*{XI.4.4}.  We indicate how
to adapt their proof to the case of finite~$n$.

Let $D = \Delta^{<n+1}$, let $i \: D \to \Delta$ be the inclusion of the
full subcategory, let $\sS$ be the category of spaces (= simplicial sets),
and let $\sS^\Delta = c\sS$ and $\sS^D$ be the functor categories of
cosimplicial spaces and $n$-truncated cosimplicial spaces, respectively.
Composition with $i$ defines the restriction functor $i^* \: c\sS \to
\sS^D$.  Let $i_* = \LKan_i \: \sS^D \to c\sS$ be the left Kan extension,
left adjoint to $i^*$.  For each $[k] \in D$, with $0 \le k \le n$,
let $D/[k]$ be the over category and let $N(D/[k])$ denote its nerve
(also known as its underlying or classifying space).  For $[k] \in D$, let
$z \: N(D/[k]) \to \Delta[k]$ be the ``zeroth vertex map'' sending each
vertex $\alpha \: [p] \to [k]$ in $N(D/[k])$ to the vertex $\alpha(0)$ in
$\Delta[k]$, see~\cite{BK72}*{XI.2.6(ii)}.  It is a weak equivalence for
each $[k]$ in $D$, since both $N(D/[k])$ and $\Delta[k]$ are contractible.
Composition with $z$ defines a morphism of mapping spaces
$$
\hom_{c\sS}(i_* i^* \Delta, Z^\bullet)
	\cong \hom_{\sS^D}(i^* \Delta, i^* Z^\bullet)
	\overset{z^*}\longto \hom_{\sS^D}(N(D/-), i^* Z^\bullet)
	= \holim_D i^* Z^\bullet \,.
$$
Here
$(i_* i^* \Delta)[q] = \colim_{[k] \to [q]} \Delta[k] \cong \sk_n \Delta[q]$
for $[q] \in \Delta$, where $[k] \to [q]$ ranges over the left fiber
category $i/[q]$ of $i$ at $[q]$.  Hence we can rewrite $z^*$ as
$$
z^* \: \Tot_n Z^\bullet = \hom_{c\sS}(\sk_n \Delta, Z^\bullet)
  \longto \holim_D i^* Z^\bullet \,.
$$
To prove that $z^*$ is a weak equivalence, we will use that $\sS^D$
with the Reedy model structure is a simplicial model category
\cite{Hir03}*{Thm.~15.3.4}.  Here $D$ has the evident Reedy category
structure inherited from~$\Delta$.  It suffices to prove that $z \:
N(D/-) \to i^* \Delta$ is a weak equivalence of cofibrant objects and
that $i^* Z^\bullet$ is a fibrant object, in the Reedy model structure on
$\sS^D$.  We have already observed that $z$ is a Reedy weak equivalence.
The cosimplicial space $\Delta$ is unaugmentable, hence cofibrant in the
model structure on $c\sS$, see~\cite{BK72}*{X.4.2}.  The latching map
of $i^* \Delta$ at each object $[k]$ in $D$ is equal to the latching
map of $\Delta$ at $[k] \in \Delta$, and therefore $i^* \Delta$ is
Reedy cofibrant.  By assumption, $Z^\bullet$ is fibrant in the model
structure on $c\sS$, see~\cite{BK72}*{X.4.6}.  The matching map of $i^*
Z^\bullet$ at each object $[k]$ in $D$ is equal to the matching map of
$Z^\bullet$ at $[k] \in \Delta$, hence $i^* Z^\bullet$ is Reedy fibrant.
It remains to check that $N(D/-)$ is Reedy cofibrant.  This follows
immediately from the fact that it is cofibrant in the projective model
structure on $\sS^D$, see~\cite{Hir03}*{Prop.~14.8.9, Thm.~11.6.1}.
\end{proof}

\subsection{Tensored structure and topological Andr{\'e}--Quillen homology}

In this subsection we specialize to the case when both $A$ and
$B$ are commutative $R$-algebras.  This category is tensored over
spaces, taking any simplicial set $S \: [p] \mapsto S_p$ and any
commutative $R$-algebra~$A$ to the realization $S \otimes_R A$ of
$[p] \mapsto S_p \otimes_R A = A \wedge_R \dots \wedge_R A$, with one
copy of~$A$ for each element of~$S_p$.  In the case when $S = S^1 =
\Delta[1]/\partial\Delta[1]$ is the simplicial circle and $R = \bS$ is
the sphere spectrum, $S^1 \otimes_{\bS} A = B^{cy}(A)$ is the cyclic bar
construction considered in the proof of Theorem~\ref{THH-TC-Cartesian},
which is equivalent to $\THH(A)$ when $A$ is flat.  Suppose now that $A$
and $B$ are $R$-com-alg-cofibrant, hence flat as $R$-modules.  Our results
about descent for $\THH$ generalize as follows.

\begin{proposition} \label{prop:descentfortensor}
If the unit map $R \to B$ is $1$-connected, then tensoring with
any simplicial set $S$ satisfies cosimplicial descent at $A$ along $R
\to B$, meaning that
$$
\eta \: S \otimes_R A \overset{\simeq}\longto
	\holim_\Delta \, S \otimes_R Y^\bullet_R(A,B)
$$
is an equivalence.
\end{proposition}

\begin{proof}
Tensors commute with coproducts, so $S \otimes_R Y^\bullet_R(A,B)$
is isomorphic to $Y^\bullet_R(S \otimes_R A, S \otimes_R B)$.
Lemma~\ref{X-is-id-cartesian} shows that $X^n_R(S_p \otimes_R A, S_p
\otimes_R B)$ is $(2n-1)$-co-Cartesian for each $p\ge0$, which implies
that $X^n_R(S \otimes_R A, S \otimes_R B)$ is $(2n-1)$-co-Cartesian and
$n$-Cartesian, for each $n\ge1$.  Hence
$$
\eta_n \: S \otimes_R A \longto
	\holim_{\Delta^{<n}} Y^\bullet_R(S \otimes_R A, S \otimes_R B)
	\simeq \holim_{P^n} X^n_R(S \otimes_R A, S \otimes_R B)
$$
is $n$-connected, and therefore $\eta$ is an equivalence.
\end{proof}

Using either the B{\"o}kstedt type model as in \cite{BCD10} (extended
mutatis mutandis to connective orthogonal or symmetric spectra) or a
model with more categorical control as in \cite{BDS}, the fundamental
homotopy cofiber sequence~\eqref{norm-restriction} leads to the
equivariant extension given below.
This generalizes our descent results for $\TC(-;p)$ to the various
forms of covering homology considered in \cite{BCD10}*{Sec.~7}.
In particular, the higher topological cyclic homology of \cite{CDD11}
satisfies cosimplicial descent at connective commutative $\bS$-algebras
along $1$-connected unit maps.

\begin{corollary}
If $G$ is a finite group acting freely on $S$, and $R \to B$ is
$1$-connected, then tensoring with $S$ satisfies equivariant cosimplicial
descent at $A$ along $R \to B$, meaning that
$$
\eta \: S \otimes_R A \overset{\simeq}\longto
	\holim_\Delta \, S \otimes_R Y^\bullet_R(A,B)
$$
is a $G$-equivariant equivalence.
\end{corollary}

\begin{proof}
We have to show that for every subgroup $H \subseteq G$, the map $\eta^H$
of $H$-fixed points is an equivalence.  This follows by essentially
the same argument as for $\THH$, e.g.~by \cite{BCD10}*{Lem.~5.1.3},
using induction over the closed families of subgroups and the fact that
homotopy orbits preserve connectivity.
\end{proof}

Coming back to the non-equivariant situation, cosimplicial descent is
satisfied by topological Andr{\'e}--Quillen homology, which we will
denote by $\TAQ^R(A)$.

\begin{proposition}
If the unit map $R \to B$ is $1$-connected, then topological
Andr\'e--Quillen homology satisfies cosimplicial descent at $A$ along
$R \to B$, in the sense that
$$
\eta \: \TAQ^R(A) \overset{\simeq}\longto
	\holim_\Delta \, \TAQ^R(Y^\bullet_R(A,B))
$$
is an equivalence.
\end{proposition}

\begin{proof}
By \cite{BM05}*{Thm.~4}, there is an equivalence
$$
\TAQ^R(A) \simeq \hocolim_m \, \Sigma^{-m} (S^m \otimes_R A)/A
$$
of $A$-module spectra, where $(S^m \otimes_R A)/A$ denotes the homotopy
cofiber of the map $A \to S^m \otimes_R A$ associated to the base point in
$S^m = (S^1)^{\wedge m}$.  The map is at least $(m-1)$-connected, for each
$m\ge0$, and likewise for $B$ in place of~$A$.  It follows that the map
$$
A \wedge_R (B/R) \wedge_R \dots \wedge_R (B/R) \longto
(S^m \otimes_R A) \wedge_R (S^m \otimes_R B)/R \wedge_R \dots \wedge_R
	(S^m \otimes_R B)/R 
$$
(with $n\ge1$ copies of~$B/R$ and of $(S^m \otimes_R B)/R$) is at least
$(m+2n-3)$-connected.  Hence the $n$-cube
$$
T \mapsto \Sigma^{-m} \, (S^m \otimes_R X^n(T))/X^n(T)
$$
(with $X^n = X^n_R(A,B)$) is $(2n-3)$-co-Cartesian, for each $m\ge0$.
Thus the $n$-cube $T \mapsto \TAQ^R(X^n(T))$
is $(2n-3)$-co-Cartesian and $(n-2)$-Cartesian, so that
$$
\eta_n \: \TAQ^R(A) \longto \holim_{P^n} \, \TAQ^R(X^n_R(A,B))
$$
is at least $(n-2)$-connected.  Passing to the homotopy limit over~$n$,
it follows that $\eta$ is an equivalence.
\end{proof}

\subsection{Less commutative examples}

Let us return to the situation where $A$ is not necessarily commutative.
In Subsection~\ref{subsec:amres} we took $B$ to be commutative to ensure
that $\mu \: B \wedge_R B \to B$ and the codegeneracy maps $s^j \: Y^{q+1}
\to Y^q$ are morphisms of $R$-algebras.  For non-commutative~$B$, we can
still define what we might call a precosimplicial $R$-algebra $[q] \mapsto
Y^q$, where now $[q]$ ranges over the subcategory $M \subset \Delta$ with
morphisms the injective, order-preserving functions.  The functor $f \:
P \to \Delta$ defined at the beginning of Subsection~\ref{subsec:cubvscos}
factors as the composite of a functor $e \: P \to M$ and the inclusion
$i \: M \to \Delta$.  Here $e$ is not left cofinal, so the analogue of
Proposition~\ref{Delta-vs-P} does not hold for general precosimplicial
spectra~$Z^\bullet$.  On the other hand, $i \: M \to \Delta$ is left
cofinal, see \cite{DD77}*{3.17}, so for cosimplicial $Y^\bullet$ the
canonical map
$$
\holim_\Delta Y^\bullet \overset{\simeq}\longto \holim_M Y^\bullet \circ i
$$
is an equivalence.  Hence
$$
\holim_M Z^\bullet \overset{\simeq}\longto \holim_P Z^\bullet \circ e
$$
is an equivalence for each precosimplicial spectrum $Z^\bullet = Y^\bullet
\circ i$ that admits an extension to a cosimplicial spectrum.

By an $\sO$ $R$-ring spectrum, for an operad~$\sO$, we mean an
$\sO$-algebra in the category of $R$-modules.  We now relax the
commutativity condition on $B$ to only ask that it is an $E_2$ $R$-ring
spectrum, i.e., an $\sO$ $R$-ring spectrum for some $E_2$ operad~$\sO$.
Then $B$ is equivalent to a monoid in a category of $A_\infty$
$R$-ring spectra, by \cite{BFV07}*{Thm.~C}, since the tensor product
of the associative operad and the little intervals operad $\sC_1$ is
an $E_2$ operad, and $\sC_1$ is an $A_\infty$ operad.  In other words,
we may assume that $\eta \: R \to B$ and $\mu \: B \wedge_R B \to B$ are
morphisms of $\sC_1$ $R$-ring spectra, so that the Amitsur resolution $A
\to Y^\bullet_R(A,B)$ is a coaugmented cosimplicial object in
the category of connective $\sC_1$ $R$-ring spectra.

Using a monadic bar construction \cite{May72}*{\S9} to functorially turn
$\sC_1$-algebras into monoids, we may replace this Amitsur resolution
with an equivalent coaugmented cosimplicial connective $R$-algebra $A
\to \bar Y^\bullet_R(A,B)$, which we might denote by $\bar Y_R(A,B)$.
Applying a homotopy functor $F$ from connective $R$-algebras to spectra,
we thus obtain a coaugmented cosimplicial spectrum $F(A) \to F(\bar
Y^\bullet_R(A,B))$.

The $\omega$-cubical diagram $\bar Y_R(A,B) \circ f_\eta$ remains
equivalent to the Amitsur $\omega$-cube associated to the $R$-module
$A$ and the unit map $\eta \: R \to B$ of $\sC_1$ $R$-ring spectra.
Replacing $B$ with an equivalent $R$-algebra~$\bar B$, we obtain an
equivalent $\omega$-cube $X^\omega_R(A, \bar B)$.  Hence there is a
chain of equivalences
$$
X^\omega_R(A, \bar B) \simeq \bar Y_R(A, B) \circ f_\eta \,.
$$
Substituting this for Lemma~\ref{X-is-Yf} in the proof of
Theorem~\ref{K-THH-TC-descent}, we obtain the following generalization
of that theorem.

\begin{theorem} \label{E2-descent}
Let $R$ be a connective commutative $\bS$-algebra, let $A$ be a connective
$R$-algebra, and let $B$ be a connective $E_2$ $R$-ring spectrum.  Suppose
that the unit map $\eta \: R \to B$ is $1$-connected.  Then the functors
$F = K$, $\THH$ and $\TC$, as well as their intermediate variants,
satisfy cosimplicial descent at $A$ along $R \to B$, in the sense that
$$
\eta \: F(A) \overset{\simeq}\longto
	\holim_\Delta F(\bar Y^\bullet_R(A,B))
$$
is an equivalence.  Here $\bar Y^q_R(A,B)$ and $Y^q_R(A,B)$ are
equivalent as $A_\infty$ $R$-ring spectra, for each $q\ge0$.
\end{theorem}

\section{Applications}

\subsection{Algebraic $K$-theory of spaces}

For each topological group $\Gamma$, the spherical group ring
$\bS[\Gamma] = \Sigma^\infty \Gamma_+$ is a connective $\bS$-algebra.
When $X \simeq B\Gamma$, the algebraic $K$-theory of $\bS[\Gamma]$ is a
model for Waldhausen's algebraic $K$-theory $A(X)$ of the space $X$,
see \cite{Wal85}.  When $X$ is a high-dimensional compact (topological,
piecewise-linear or differentiable) manifold, $A(X)$ is closely related
to the space of $h$-cobordisms on $X$ and the group of automorphisms of
$X$ \cite{WJR13}.  This motivates the interest in the algebraic $K$-theory
of $\bS$ and the associated spherical group rings.

Base change along the Hurewicz map $\bS \to H\bZ$ induces rational
equivalences $\bS[\Gamma] \to H\bZ[\Gamma]$ and $K(\bS[\Gamma]) \to
K(H\bZ[\Gamma])$.  In particular, $A(*) = K(\bS) \to K(H\bZ) = K(\bZ)$
is a rational equivalence, and Borel's rational computation of the
algebraic $K$-theory of the integers \cite{Bor74} gives strong rational
information about the $h$-cobordism spaces and automorphism groups of
high-dimensional highly-connected manifolds \cite{Igu88}*{p.~7}.

\subsection{Descent along $\bS \to H\bZ$}

To obtain torsion information about $A(X) = K(\bS[\Gamma])$ one can instead
consider the cosimplicial resolution
$$
\xymatrix{
\bS \ar[r]^-{\eta} & H\bZ \ar@<0.8ex>[r] \ar@<-0.8ex>[r]
        & H\bZ \wedge H\bZ \ar[l] \ar@<1.6ex>[r] \ar[r] \ar@<-1.6ex>[r]
        & \ \cdots \ar@<0.8ex>[l] \ar@<-0.8ex>[l]
}
$$
in the category of connective $\bS$-algebras, and the induced coaugmented
cosimplicial spectrum
$$
\xymatrix{
K(\bS[\Gamma]) \ar[r]^-{\eta} & K(H\bZ[\Gamma]) \ar@<0.8ex>[r] \ar@<-0.8ex>[r]
        & K((H\bZ \wedge H\bZ)[\Gamma]) \ar[l] \ar@<1.6ex>[r] \ar[r] \ar@<-1.6ex>[r]
        & \ \cdots \ar@<0.8ex>[l] \ar@<-0.8ex>[l] \,.
}
$$
By Theorem~\ref{K-THH-TC-descent} the natural map from $K(\bS[\Gamma])$
to the homotopy limit of this cosimplicial spectrum is an equivalence,
and similarly for $\THH$ and $\TC$.  These cases of descent, along
$\bS \to H\bZ$, are essentially those studied in \cite{Dun97}. See also
\cite{Tsa00} for descent results in the context of commutative rings.

A computational drawback with this approach is the structure of the
smash product
$$
(H\bZ \wedge \dots \wedge H\bZ)[\Gamma]
	= \bS[\Gamma] \wedge H\bZ \wedge \dots \wedge H\bZ
$$
with $(q+1)$ copies of~$H\bZ$.  It is a connective $H\bZ$-algebra,
hence equivalent to a simplicial ring, but for $q\ge1$ the algebraic
$K$-theory and topological cyclic homology of this simplicial ring appear
to be difficult to analyze.

\subsection{Descent along $\bS \to MU$}

Experience from algebraic topology shows that the complex bordism spectrum
$MU$ is a convenient stopping point on the way from the sphere spectrum
to the integers:
$$
\bS \longto MU \longto H\bZ \,.
$$
Here $MU$ is a commutative $\bS$-algebra with $1$-connected unit map $\bS
\to MU$.  The coefficient ring $MU_* = \pi_*(MU) = \bZ[x_k \mid k\ge1]$
and the homology algebra $H_*(MU) \cong H_*(BU) = \bZ[b_k \mid k\ge1]$
are explicitly known \cite{Mil60}, \cite{Nov62}, with $|x_k| = |b_k| = 2k$
for each~$k$.  The associated cosimplicial resolution
$$
\xymatrix{
\bS \ar[r]^-{\eta} & MU \ar@<0.8ex>[r] \ar@<-0.8ex>[r]
        & MU \wedge MU \ar[l] \ar@<1.6ex>[r] \ar[r] \ar@<-1.6ex>[r]
        & \ \cdots \ar@<0.8ex>[l] \ar@<-0.8ex>[l]
}
$$
induces the coaugmented cosimplicial spectrum
$$
\xymatrix{
K(\bS[\Gamma]) \ar[r]^-{\eta} & K(MU[\Gamma]) \ar@<0.8ex>[r] \ar@<-0.8ex>[r]
        & K((MU \wedge MU)[\Gamma]) \ar[l] \ar@<1.6ex>[r] \ar[r] \ar@<-1.6ex>[r]
        & \ \cdots \ar@<0.8ex>[l] \ar@<-0.8ex>[l] \,.
}
$$
By Theorem~\ref{K-THH-TC-descent}, applied with $R = \bS$, $A =
\bS[\Gamma]$ and $B = MU$, the natural map from $K(\bS[\Gamma])$ to the
homotopy limit of this cosimplicial spectrum is an equivalence, and there
are corresponding equivalences for $\THH$ and $\TC$.

From its definition as an $E_\infty$ Thom spectrum, $MU$ comes equipped
with a Thom equivalence $MU \wedge MU \simeq MU \wedge BU_+$.  Hence the
smash product
$$
(MU \wedge \dots \wedge MU)[\Gamma]
	\simeq \bS[\Gamma] \wedge MU \wedge BU^q_+
$$
that occurs in codegree~$q$ is not significantly more complicated for
$q\ge1$ than for $q=0$.  This means that it may be more realistic to
study $K(\bS)$ by descent along $\bS \to MU$ than by descent along
$\bS \to H\bZ$.  We propose that in order to understand $A(*) = K(\bS)$
from the chromatic point of view \cite{MRW77}, \cite{Rav84}, one should
study this cosimplicial resolution, starting with $K(MU)$ and continuing
with $K(MU \wedge BU^q_+)$ for each $q\ge0$, together with the associated
descent spectral sequence converging to $\pi_* K(\bS)$.  Similar remarks
apply for $A(X) = K(\bS[\Gamma])$ and for the functors $\THH$ and $\TC$,
see Theorem~\ref{descentspseq}.

\subsection{(Hopf-)Galois descent}

In the language of \cite{Rog08}*{\S12}, the map $\bS \to MU$ is a
Hopf--Galois extension of commutative $\bS$-algebras.  Our result can be
viewed as proving $1$-connected Hopf--Galois descent for $K$, $\THH$ and
$\TC$, but the actual Hopf $\bS$-algebra coaction plays no role in the proof.
Analogously, consider a $G$-Galois extension $A \to B$ of commutative
$\bS$-algebras, in the sense of \cite{Rog08}*{\S4}, with $G$ a finite group.
The equivalence $B \wedge_A B \simeq F(G_+, B) \cong F(G^2_+, B)^G$
induces a level equivalence of cosimplicial resolutions from
$$
\xymatrix{
A \ar[r]^-{\eta} & B \ar@<0.8ex>[r] \ar@<-0.8ex>[r]
        & B \wedge_A B \ar[l] \ar@<1.6ex>[r] \ar[r] \ar@<-1.6ex>[r]
        & \ \cdots \ar@<0.8ex>[l] \ar@<-0.8ex>[l]
}
$$
to
$$
\xymatrix{
A \ar[r]^-{\eta} & F(G_+, B)^G \ar@<0.8ex>[r] \ar@<-0.8ex>[r]
        & F(G^2_+, B)^G \ar[l] \ar@<1.6ex>[r] \ar[r] \ar@<-1.6ex>[r]
        & \ \cdots \ar@<0.8ex>[l] \ar@<-0.8ex>[l] \,,
}
$$
i.e., from the Amitsur resolution $Y^\bullet = Y^\bullet_A(A,B)$ to the
cosimplicial commutative $\bS$-algebra $F(E_\bullet G_+, B)^G$ obtained
by mapping out of the free contractible simplicial $G$-set $E_\bullet G \:
[q] \mapsto E_q G = G^{q+1}$.

Algebraic $K$-theory preserves equivalences and commutes with finite
products, so the cosimplicial spectra $K(Y^\bullet)$, $K(F(E_\bullet G_+,
B)^G)$ and $F(E_\bullet G_+, K(B))^G$ are level equivalent.  Hence the
homotopy limit of $K(Y^\bullet)$ is equivalent to the homotopy fixed
point spectrum $F(EG_+, K(B))^G = K(B)^{hG}$, where $EG = |E_\bullet G|$.
(The homotopy type of $K(B)^{hG}$ only depends on the homotopy type of
$K(B)$ as a spectrum with $G$-action, and the $G$-action is determined
by functoriality.)
The canonical map $\eta \: K(A) \to K(B)^{hG}$ is not in general an
equivalence, so algebraic $K$-theory does not in general satisfy descent
along Galois extensions $A \to B$.  However, a recent result of Clausen,
Mathew, Naumann and Noel \cite{CMNN}*{Thm.~1.7} shows that, after
any ``periodic localization'', algebraic $K$-theory satisfies descent
along all maps $A \to B$ of commutative $\bS$-algebras for which $B$ is
dualizable as an $A$-module and the restriction map $K_0(B) \to K_0(A)$
is rationally surjective.  Again, the Galois condition plays no role
for their proof.  See also \cite{Tho85} for the corresponding result for
Bott localized algebraic $K$-theory of commutative rings (or schemes).

\subsection{Descent along $\bS \to X(n)$}

There is a sequence of $E_2$ ring spectra $X(n)$ interpolating between
$\bS$ and $MU$, see \cite{DHS88}.  Recall that $MU = BU^\gamma$ is the
Thom spectrum of a virtual vector bundle~$\gamma$ over~$BU$.  For each
$n\ge2$, the Thom spectrum $X(n) = \Omega SU(n)^\gamma$, of the pullback
of $\gamma$ over the double loop map $\Omega SU(n) \to \Omega SU \simeq
BU$, is an $E_2$ ring spectrum with $1$-connected unit map $\bS \to X(n)$.
There are natural maps of $E_2$ ring spectra
$$
\bS \longto \dots \longto X(n) \longto \dots \longto MU \longto H\bZ
$$
connecting these examples to those previously discussed.  We identify
$H_*(X(n)) \cong H_*(\Omega SU(n))$ with the subalgebra $\bZ[b_1, \dots,
b_{n-1}]$ of $H_*(MU) \cong H_*(BU) =  \bZ[b_k \mid k\ge1]$.

By Theorem~\ref{E2-descent}, the functors $K$, $\THH$ and $\TC$ satisfy
descent along $\bS \to X(n)$ for each $n\ge2$.  There are Thom equivalences
$X(n) \wedge X(n) \simeq X(n) \wedge \Omega SU(n)_+$, so the study
of $K(\bS)$ by descent along $\bS \to X(n)$ leads to the study of the
algebraic $K$-theory of $X(n) \wedge (\Omega SU(n))^q_+$ for $q\ge0$,
and similarly for $K(\bS[\Gamma])$, $\THH$ and $\TC$.
The $E_2$ ring spectra $X(n)$ are closer to $\bS$ than $MU$, hence
$K(X(n))$ can yield finer information about $K(\bS)$ than $K(MU)$ does.
However, like in the case of $\bS$, the homotopy groups of $X(n)$ are not
explicitly known, so a direct analysis of $\pi_* \THH(X(n))$ and $\pi_*
\TC(X(n))$ may be less feasible than in the case of $MU$.

\subsection{Trace methods}

The cyclotomic trace map $\trc \: K(B) \to \TC(B;p)$ introduced in
\cite{BHM93}, in conjunction with the relative equivalence theorem from
\cite{Dun97}, is the main method available for calculating the algebraic
$K$-groups of connective $\bS$-algebras other than (simplicial) rings.
In the case of the sphere spectrum, $\TC(\bS;p)$ is $p$-adically
equivalent to $\bS \vee \Sigma \bC P^\infty_{-1}$, so calculations of
$\pi_* K(\bS)$ are possible in a moderate range of degrees \cite{Rog02},
\cite{Rog03} (see also recent work of Blumberg--Mandell \cite{BM} at
irregular primes).  Nonetheless, complete calculations are at least as
hard as those for $\pi_*(\bS)$, hence appear to be out of reach.

The difficulty of understanding the stable homotopy groups of spheres
can be formulated as the difficulty of understanding the Adams--Novikov
spectral sequence \cite{Nov67}
\begin{equation} \label{ANSS}
E_2^{s,t} = \Ext_{MU_*MU}^{s,t}(MU_*, MU_*)
	\Longrightarrow_s \pi_{t-s}(\bS),
\end{equation}
i.e., to understand the descent spectral sequence
$$
E_2^{s,t} = \pi^s \pi_t Y^\bullet
	\Longrightarrow_s \pi_{t-s} \holim_\Delta Y^\bullet
$$
associated with the cosimplicial commutative $\bS$-algebra $Y^\bullet =
Y^\bullet_\bS(\bS,MU)$, with $Y^q = MU \wedge \dots \wedge MU \simeq
MU \wedge BU^q_+$.  An advantage of this approach is that chromatic
phenomena in $\pi_*(\bS)$ are more readily visible at the $E_2$-term of
the Adams--Novikov spectral sequence.

By analogy, the difficulty of understanding $\pi_* K(\bS)$ and
{$\pi_*\TC(\bS; p)$} can be separated into two parts: first that of
understanding the cosimplicial objects $[q] \mapsto \pi_* K(Y^q)$ and
$[q] \mapsto \pi_* \TC(Y^q; p)$, and secondly that of understanding the
behavior of the descent spectral sequences
$$
E_2^{s,t} = \pi^s \pi_t K(Y^\bullet)
	\Longrightarrow_s \pi_{t-s} K(\bS)
$$
and
$$
E_2^{s,t} = \pi^s \pi_t \TC(Y^\bullet; p)
	\Longrightarrow_s \pi_{t-s} \TC(\bS; p) \,.
$$
The first aim of understanding $\pi_* K(MU)$ and $\pi_* \TC(MU;
p)$, corresponding to $q=0$, then plays an analogous role to that of
understanding $MU_* = \pi_*(MU)$.  An optimist may seek to discern
chromatic phenomena in $\pi_* K(\bS)$ and $\pi_* \TC(\bS; p)$ at the
level of these $E_2$-terms.

\subsection{Descent for $\THH$}

As an illustration, Theorem~\ref{descentspseq} for $\THH$ at $\bS$
along $\bS \to MU$ gives a strongly convergent descent spectral sequence
$$
E_2^{s,t} = \pi^s \pi_t \THH(Y^\bullet)
	\Longrightarrow_s \pi_{t-s} \THH(\bS) \,.
$$
There is an equivalence $\THH(MU) \simeq MU \wedge SU_+$, by
\cite{BCS10}*{Cor.~1.1}, and the Atiyah--Hirzebruch spectral sequence
$E^2_{*,*} = H_*(SU; MU_*) \Longrightarrow \pi_*(MU \wedge SU_+)$
collapses at $E^2$ to give the algebra isomorphism
$$
\pi_* \THH(MU) \cong MU_* \otimes E(e_k \mid k\ge1) = LE
$$
of \cite{MS93}*{Rmk.~4.3}.  Here $L = MU_* = \bZ[x_k \mid k\ge1]$ is
the Lazard ring, $E = E(e_k \mid k\ge1)$ is the exterior algebra over
$\bZ$ on a sequence of generators $e_k$, with $|e_k| = 2k+1$, and $LE =
L \otimes E$ is shorthand notation for their tensor product.  Similarly,
\begin{align*}
MU_* \THH(MU) &\cong \pi_*(MU \wedge MU \wedge_{MU} \THH(MU)) \\
	&\cong LB \otimes_L LE \cong LBE
\end{align*}
is flat as a left $MU_*$-module, where $LB = MU_* MU$ and $B = H_*(MU)
\cong \bZ[b_k \mid k\ge1]$, and
\begin{align*}
\pi_* \THH(MU \wedge MU)
        &\cong \pi_*( \THH(MU) \wedge_{MU} \wedge MU \wedge \THH(MU) ) \\
	&\cong LE \otimes_L LBE \cong LE \otimes BE \,.
\end{align*}
In general, $\pi_* \THH(Y^q) \cong LE \otimes (BE)^{\otimes q}$, and $[q]
\mapsto \pi_* \THH(Y^q)$ is the cobar construction associated to the
split Hopf algebroid $(LE, LE \otimes BE)$, see~\cite{Rav86}*{App.~A1}.
Hence the descent spectral sequence for $\THH$ takes the form
\begin{equation} \label{thhdescent}
E_2^{s,t} = \Ext^{s,t}_{LE \otimes BE}(LE, LE)
	\Longrightarrow_s \pi_{t-s} \THH(\bS) \,.
\end{equation}
The unit inclusion $\bZ \to E$ induces an equivalence 
$$
(MU_*, MU_* MU) = (L, LB) \longto (LE, LE \otimes BE)
$$
of Hopf algebroids, which induces an isomorphism from the 
Adams--Novikov spectral sequence~\eqref{ANSS} to the descent spectral
sequence~\eqref{thhdescent}.  This is of course compatible with the
identity $\bS = \THH(\bS)$, and shows, somewhat tautologically, that
the descent spectral sequence for $\THH$ at $\bS$ along $\bS \to MU$
has an $E_2$-term that is susceptible to chromatic analysis along the
lines of \cite{MRW77} and~\cite{Rav86}*{Ch.~5}.

By contrast, descent for $\THH$ at $\bS$ along $\bS \to H\bZ$ leads to
the study of $\pi_* \THH(H\bZ \wedge \dots \wedge H\bZ) \cong \pi_*(
\THH(\bZ) \wedge \dots \wedge \THH(\bZ) )$, with $(q+1)$ copies of
$H\bZ$ or $\THH(\bZ)$, and while $\pi_* \THH(\bZ)$ is explicitly known
\cite{BM94}*{\S5}, the term $\pi_*( \THH(\bZ) \wedge \THH(\bZ) )$ is not
flat over $\pi_* \THH(\bZ)$, and there is no description of the resulting
$E_1$-term as the cobar complex of a Hopf algebroid.  If we instead were
to study the functor $F(A) = \THH(A) \wedge S/p$, where $S/p$ is the
mod~$p$ Moore spectrum, then $\pi_* F(H\bZ \wedge H\bZ) = \pi_*(\THH(\bZ)
\wedge \THH(\bZ); \bZ/p)$ is a flat Hopf algebroid over $\pi_* F(H\bZ)
= \pi_*(\THH(\bZ); \bZ/p)$.  The associated cobar complex is isomorphic
to the $E_1$-term of the descent spectral sequence for $F$ at $\bS$ along
$\bS \to \bZ$, which in turn is isomorphic to the $\THH(\bZ)$-based Adams
spectral sequence for $S/p$, see \cite{MNN17}*{Prop.~2.14}.  The canonical
$\THH(\bZ)$-based tower for $S/p$ is also a mod~$p$ Adams tower for $S/p$,
so these spectral sequences are therefore isomorphic to the mod~$p$
Adams spectral sequence for $S/p$, from the $E_2$-term and onwards.

\subsection{Fixed points of $\THH$}

The next steps toward analyzing descent for $\TC$ along $\bS \to MU$
would be to study descent for the fixed point functor $\THH(-)^C$
along $\bS \to MU$, for each $C = C_{p^m}$.  Let $\THH(A)^{tC} =
[\widetilde{EC} \wedge F(EC_+, \THH(A))]^C$ denote the $C$-Tate
construction on $\THH(A)$, i.e., the $C$-fixed points of the Tate
$C$-spectrum of \cite{GM95}*{pp.~3-4}.  The comparison maps
\begin{align*}
\Gamma_m &\: \THH(MU)^{C_{p^m}} \longto \THH(MU)^{hC_{p^m}} \\
\hat\Gamma_m &\: \THH(MU)^{C_{p^{m-1}}} \longto \THH(MU)^{tC_{p^m}}
\end{align*}
are known to be equivalences after $p$-completion, by
\cite{LNR11}*{Thm.~1.1} in the case $m=1$, hence also for the cases
$m\ge2$ by \cite{Tsa98}*{Thm.~2.4} or \cite{BBLNR14}*{Thm.~2.8}.
It is very likely that the same methods will apply when $Y^0 = MU$ is replaced
by $Y^q = MU \wedge \dots \wedge MU$ for $q\ge1$.  If so, the
$p$-completed spectral sequence
$$
E_2^{s,t} = \pi^s \pi_t \THH(Y^\bullet)^{C_{p^m}}_p
	\Longrightarrow_s \pi_{t-s} \THH(\bS)^{C_{p^m}}_p
$$
can equally well be analyzed using $C_{p^m}$-homotopy fixed points or
$C_{p^{m+1}}$-Tate constructions in place of $C_{p^m}$-fixed points.
A first step in this direction would be to determine the differential
structure of the $C_p$-Tate spectral sequence
$$
\hat E^2_{s,t} = \hat H^{-s}(C_p; \pi_t \THH(MU))
	\Longrightarrow_s \pi_{s+t} \THH(MU)^{tC_p} \,.
$$
By truncation to the second quadrant, this would determine the
$C_p$-homotopy fixed point spectral sequence converging to $\pi_*
\THH(MU)^{hC_p}$, hence also $\pi_* \THH(MU)^{C_p}$ after $p$-completion.

The comparison maps $\Gamma_m \: \THH(\bS)^{C_{p^m}} \to
\THH(\bS)^{hC_{p^m}}$ and $\hat\Gamma_m \: \THH(\bS)^{C_{p^{m-1}}} \to
\THH(\bS)^{tC_{p^m}}$ are also $p$-adic equivalences, by the proven Segal
conjecture \cite{Car84}, but even in the case $m=1$ the differential
structure of the $C_p$-Tate spectral sequence converging to $\pi_*
\THH(\bS)^{tC_p} \cong \pi_*(\bS)$ is not known \cite{Ada74}.
In the case of $\THH(\bZ)$ the maps $\Gamma_m$ and $\hat\Gamma_m$ are
not quite equivalences, but they do induce isomorphisms in homotopy
with mod~$p$ coefficients in non-negative degrees.  The structure
of the $C_{p^m}$-Tate spectral sequence converging to $\pi_*(
\THH(\bZ)^{tC_{p^m}}; \bZ/p)$ is known \cite{BM95}, \cite{Rog99}, but
these calculations appear to be difficult to extend to the case of
$\THH(H\bZ \wedge \dots \wedge H\bZ)$.

\subsection{Topological periodic homology}

The limiting maps
\begin{align*}
\Gamma &\: \TF(MU; p) \longto \holim_m \THH(MU)^{hC_{p^m}}
	\longleftarrow \THH(MU)^{h\bT} \\
\hat\Gamma &\: \TF(MU; p) \longto \holim_m \THH(MU)^{tC_{p^m}}
	\longleftarrow \THH(MU)^{t\bT}
\end{align*}
are also equivalences after $p$-completion, cf.~\cite{BM95}*{(2.11)}
and~\cite{AR02}*{Thm.~5.7}.  Thus the $E^2$-terms of the $\bT$-homotopy
fixed point and $\bT$-Tate spectral sequences
\begin{align*}
E^2_{*,*} &= \bZ[t] \otimes \pi_* \THH(MU)
	\Longrightarrow \pi_* \THH(MU)^{h\bT} \\
\hat E^2_{*,*} &= \bZ[t^{\pm1}] \otimes \pi_* \THH(MU)
	\Longrightarrow \pi_* \THH(MU)^{t\bT}
\end{align*}
are known, and both converge to $\pi_* \TF(MU; p)$ after $p$-completion.
Note that $\TP(MU) = \THH(MU)^{t\bT}$ is the arithmetically interesting
topological periodic homology studied by Hesselholt \cite{Hes}, also
known as periodic topological cyclic homology or topological de\,Rham
homology.
As in the case of $\TF(\bS; p)$, the inverse Frobenius operator
$\varphi^{-1} = R \: \TF(MU; p) \to \TF(MU; p)$ extends to $\TP(MU)$
after $p$-completion, without any further localization.

\subsection{Rational analysis}

The rational algebraic $K$-groups of $MU$, i.e., $\pi_* K(MU) \otimes
\bQ$, were determined in \cite{AR12}*{Thm.~4.2} by using Goodwillie's
theorem \cite{Goo86}.  The Poincar{\'e} series is
$$
\frac{x^5}{1-x^4} + \frac{1 + x h(x)}{1+x}
= 1 + x^3 + 3 x^5 + 3 x^7 + x^8 + 6 x^9 + 2 x^{10} + \dots
$$
where
$$
h(x) = \prod_{k\ge1} \frac{1 + x^{2k+1}}{1 - x^{2k}} \,.
$$
A similar result holds for each $\pi_* K(Y^q) \otimes \bQ$, where $Y^q =
MU \wedge \dots \wedge MU$, with $(q+1)$ copies of~$MU$.  Since $\pi_0
Y^q = \bZ$ and $Y^q$ is connective and of finite type, we know
by an easy generalization of \cite{Dwy80}*{Prop.~1.2} that
$E_1^{s,t} = \pi_t K(Y^s)$ is a finitely generated abelian group in
each bidegree~$(s,t)$.  By the following result $E_2^{s,t} = \pi^s \pi_t
K(Y^\bullet)$ is in fact finite in each bidegree, except at the edge
$s=0$, where $E_2^{0,t}$ is rationally isomorphic to $\pi_t K(\bZ)$.

\begin{proposition}
The rationalized descent spectral sequence
$$
E_2^{s,t} \otimes \bQ = \pi^s \pi_t K(Y^\bullet) \otimes \bQ
	\Longrightarrow_s \pi_{t-s} K(\bS) \otimes \bQ \,,
$$
for algebraic $K$-theory at $\bS$ along $\bS \to MU$,
collapses at the $E_2$-term to the edge $s=0$.
\end{proposition}

\begin{proof}
The vanishing line and strong convergence of the descent spectral
sequence $E_2^{s,t} \Longrightarrow_s \pi_{t-s} K(\bS)$ from
Theorem~\ref{descentspseq} imply the same vanishing line and strong
convergence for the rationalized spectral sequence.

The unit map $\bS \to Y^q$ and the zeroth Postnikov section $Y^q \to
H\bZ$ induce maps $K(\bS) \to K(Y^\bullet) \to K(\bZ)$ that, after
rationalization, split off a copy of $K(\bS)$ from $K(Y^\bullet)$.
It remains to prove that the remainder of the rationalized spectral
sequence, associated to the cosimplicial spectrum with the homotopy
fiber of $K(Y^q) \to K(\bZ)$ in codegree~$q$, collapses to zero at
the $E_2$-term.

The Hurewicz homomorphism $MU_*=\pi_*(MU) \to H_*(MU) \cong B$ is a
rational equivalence, so $\pi_* \THH(MU)$ is rationally isomorphic
to $HH_*(B) = BE$, where $B = \bZ[b_k \mid k\ge1]$ and $E = E(e_k \mid
k\ge1)$.  Connes' $B$-operator on $HH_*(B)$ corresponds to the suspension
operator $\sigma$, which is the differential and derivation given by
$\sigma(b_k) = e_k$ for each $k\ge1$.  Hence the rationalized de\,Rham
homology $H^{dR}_*(B) \otimes \bQ = H_*(BE, \sigma) \otimes \bQ = \bQ$
is trivial in positive degrees.  By \cite{AR12}*{Cor.~2.4}, which is
a consequence of \cite{Goo86}*{II.3.4}, it follows that the trace map
$K(MU) \to \THH(MU)$ identifies the kernel of $\pi_* K(MU) \to \pi_*
K(\bZ)$ with the image of $\sigma \: BE \to BE$, after rationalization.

Similarly, $\pi_* \THH(Y^\bullet)$ is rationally isomorphic to
the cosimplicial resolution $[q] \mapsto HH_*(B^{\otimes q+1})
\cong (BE)^{\otimes q+1}$ associated to $\eta \: \bZ \to BE$.
By \cite{AR12}*{Cor.~2.4} again, the trace map identifies the kernel of
$\pi_* K(Y^q) \to \pi_* K(\bZ)$ with the image of $\sigma \: (BE)^{\otimes
q+1} \to (BE)^{\otimes q+1}$, after rationalization.  Let $(C^*,
\delta)$ and $(D^*, \delta)$ be the associated cochain complexes,
with $C^q = (BE)^{\otimes q+1}$ and $D^q = \im(\sigma) \subset C^q$.
The $E_2$-term of the remainder of the descent spectral sequence is
given by the cohomology of $(D^*, \delta)$, after rationalization.

The augmentation $\epsilon \: BE \to \bZ$ induces a cochain contraction
$\epsilon \otimes id_{BE}^{\otimes q}$ of $\eta \: \bZ \to C^*$, mapping
$x_0 \otimes x_1 \otimes \dots \otimes x_q \in C^q$ to $\epsilon(x_0) x_1
\otimes \dots \otimes x_q \in C^{q-1}$.  It restricts to a contraction
of $0 \to D^*$, since $\epsilon \otimes id_{BE}^{\otimes q}$ commutes
with $\sigma$.  Hence the cohomology of $(D^*, \delta)$ is zero, as claimed.
\end{proof}

\end{document}